\documentclass[11pt]{amsart}
\usepackage{amssymb,amsmath,latexsym,enumerate, dsfont}
\usepackage{graphicx,verbatim,enumerate,%
ifpdf,mdwlist}
\usepackage{color}
\ifpdf
\usepackage{hyperref}
\else
\usepackage[hypertex]{hyperref}
\fi
\usepackage{graphicx}
\usepackage{cancel}
\usepackage{float}
\renewcommand{\phi}{\varphi}

\newtheorem{theorem}{Theorem}[section]
\newtheorem{thm}{Theorem}[section]

\newtheorem{thesis}[theorem]{Thesis}
\newtheorem{statement}[theorem]{Statement}
\newtheorem{lemma}[theorem]{Lemma}

\newtheorem{corollary}[theorem]{Corollary}

\newtheorem{question}[theorem]{Question}

\newtheorem{defi}[theorem]{Definition}
\newenvironment{definition}{\begin{defi} \rm}{ \end{defi}}
\newenvironment{defn}{\begin{defi} \rm}{ \end{defi}}
\newtheorem{exa}[theorem]{Example}

\newtheorem{rem}[theorem]{Remark}

\DeclareMathOperator{\tp}{tp}

\newcommand{\N}{\mathbb{N}}

\def \l{\mathcal{L}}
\def \Th{\operatorname{Th}}
\def \m{\mathcal{M}}
\def \u{\mathcal{U}}
\def \p{\mathcal{P}}
\def \n{\mathcal{N}}

\def \FP{\operatorname{FP}}

\title[Hindman's Theorem re-visited]{Hindman's theorem and idempotent types}

\author[U.~Andrews]{Uri Andrews}
\address{Department of Mathematics\\
University of Wisconsin\\
Madison, WI 53706-1388\\
USA}
\email{\href{mailto:andrews@math.wisc.edu}{andrews@math.wisc.edu}}
\urladdr{\url{http://www.math.wisc.edu/~andrews/}}

\author[I.~Goldbring]{Isaac Goldbring}
\address{Department of Mathematics, Statistics, and Computer Science, University of Illinois at Chicago, Science and Engineering Offices M/C 249, 851 S. Morgan St.\\
Chicago, IL, 60607}
\email{\href{mailto:isaac@math.uic.edu}{isaac@math.uic.edu}}
\urladdr{\url{homepages.math.uic.edu/~isaac/}}

\begin{document}

\begin{abstract}
Motivated by a question of Di Nasso, we show that Hindman's Theorem is equivalent to the existence of idempotent types in countable complete extensions of Peano Arithmetic.
\end{abstract}

\maketitle

\section{Introduction}

Recall that $X\subseteq \N$ is said to be an \emph{IP set} if there is infinite $Y\subseteq X$ such that every finite sum of distinct elements of $Y$ is in $X$.  Hindman's Theorem asserts that if $\N$ is partitioned into finitely many pieces, then one of the pieces is an IP set.

Hindman's original proof was very combinatorial in nature.  Later, Galvin and Glazer gave a ``soft" proof of Hindman's theorem using the notion of an \emph{idempotent ultrafilter}.  Recall that an ultrafilter $\u$ on $\N$ is said to idempotent if, for all $A\subseteq \N$, we have $$A\in \u \Leftrightarrow \{n\in \N \ : \ A-n\in \u\}\in\u;$$ here, $A-n:=\{x\in \N \ : \ x+n\in A\}$.  It is readily verified that all sets in an idempotent ultrafilter are IP sets, so to establish Hindman's theorem, it suffices to establish the existence of an idempotent ultrafilter.  This latter task can be accomplished via several applications of Zorn's lemma and essentially boils down to Ellis' theorem about compact semi-topological semigroups.

In \cite{dinasso}, Di Nasso asks whether or not there can be a ``nonstandard" proof of the existence of idempotent ultrafilters, presumably using only the same amount of choice needed to prove the existence of ordinary nonprincipal ultrafilters.  In order to formulate an attack on this problem, he establishes a purely model-theoretic formulation of the existence of idempotent ultrafilters:  there exists $\alpha,\beta \in \N^*$ satisfying the following two properties:
\begin{itemize}
\item for all $A\subseteq \N$, we have $\alpha\in A^*\Leftrightarrow \beta \in A^*\Leftrightarrow \alpha+\beta\in A^*$;
\item for all $B\subseteq \N^2$, if $(\alpha,\beta)\in B^*$, then there is $n\in \N$ such that $(n,\beta)\in B^*$.
\end{itemize}

By replacing $A$ and $B$ by 0-definable sets relative to some complete theory $T$ extending Peano Arithmetic (PA), it now makes sense to talk about \emph{idempotent types} in such theories.  In \cite{dinasso}, Di Nasso also asks for a sufficient condition to guarantee the existence of idempotent types in arbitrary complete extensions of PA (with an eye towards an answer to his earlier question).  The main result of this note is that Hindman's theorem is actually equivalent to the existence of idempotent types in arbitrary \emph{countable} complete extensions of PA; in particular, idempotent types always exist in such theories.  (We actually use the version of Hindman's theorem that states that the family of IP sets is \emph{partition regular}, meaning that if $X\subseteq \N$ is an IP set and $Y$ is a subset of $X$, then either $Y$ or $X\setminus Y$ is an IP set.  Accordingly, we show that idempotent types containing a prescribed 0-definable IP set always exist.)

It is not clear to us if the existence of idempotent types in countable complete extensions of PA can be used to obtain idempotent ultrafilters using some sort of compactness argument.  Conversely, it is not clear to us how to use idempotent ultrafilters to obtain idempotent types. 

Hindman's theorem and idempotent ultrafilters actually make sense in the much more general context of semigroups and so we prove all of our results in this more general context.

\subsection{Constructive consequences}

In \cite{tao}, Tao alludes to the fact that arguments in combinatorics involving idempotent ultrafilters are highly nonconstructive as one needs to use the axiom of choice multiple times to prove the existence of an idempotent ultrafilter.  (As a curiosity, it is not currently known whether or not the existence of idempotent ultrafilters on $\mathbb N$ is equivalent, over ZF, to the existence of nonprincipal ultrafilters on $\mathbb N$.)  In relation to this fact, here is a vague conjecture:

\begin{thesis}
Any argument in combinatorics utilizing the existence of idempotent ultrafilters could instead use idempotent types.
\end{thesis}

If this thesis is true, then by our main result, any argument in combinatorics utilizing idempotent ultrafilters could instead use Hindman's theorem directly; since Hindman's theorem can be proven constructively and our proof that Hindman's theorem implies the existence of idempotent types is also constructive, this would allow all arguments using idempotent ultrafilters to be made constructive.  It is not clear how one could prove (or even precisely formulate) the above thesis, but, intuitively speaking, in proving a result about a set $A$ of natural numbers, the argument involved should only mention sets definable from $A$ in second order arithmetic and so an idempotent type in an appropriate countable language should suffice to carry out the argument.

Concerning the reverse mathematical strength of our result, we show that the existence of idempotent types is enough to carry out the usual ultrafilter proof of Hindman's theorem using only RCA$_0$.  Conversely, proving the existence of idempotent types from Hindman's theorem seems to need to use $\Pi^1_1$-comprehension.



\section{Definitions}

By a \emph{semigroup structure} we mean a first-order structure $\m:=(M,\cdot,\ldots)$ in a countable language such that $(M,\cdot)$ is a semigroup; in this case, we say that $\m$ is \emph{based on} $(M,\cdot)$.

\begin{defn}
We say $q(x,y)\in S_2(M)$ is an \emph{independent type} if, for any $\varphi(x,y)\in q$, there is $u\in M$ such that $\varphi(u,y)\in q$.
\end{defn}

%

\begin{rem}
In model-theoretic lingo, independent types are simply heirs.  More precisely, if $(a,b)$ realizes $q$ (in some elementary extension of $\m$), then $q$ is independent if and only if $\tp(b/Ma)$ is an heir of $\tp(b/M)$.
\end{rem}

\begin{defn}
$p(x)\in S_1(M)$ is called an \emph{idempotent type} if there is an independent type $q(x,y)$ such that $p(x),p(y),p(x\cdot y)\subseteq q(x,y)$.
\end{defn}

\begin{rem}
In the definition of idempotent type, we do not insist that the type be non-principal.  In fact, an idempotent type $p(x)\in S(M)$ is principal if and only if $p(x)=\tp(a/M)$ for $a\in M$ idempotent.  We will have more to say about this at the end of the paper.
\end{rem}


\begin{rem}
Recall that the (model-theoretic) \emph{completion} of $\N$ is the structure $\N^\#$ with a symbol for every function and relation on $\N$ and a symbol for every element of $\N$.  In \cite{dinasso}, it is shown that if $T^\#:=\Th(\N^\#)$, then idempotent types for $T^\#$ are precisely the idempotent ultrafilters on $\N$.  The same observation (with an identical proof) actually holds for arbitrary semigroup structures.
\end{rem}


\begin{defn}
Let $(M,\cdot)$ be a semigroup.  If $(u_n)$ is a countable sequence from $M$, we define $\FP(u_n):=\{u_{i_1}\cdots u_{i_k} \ : \ i_1<\cdots <i_{k}\}$.  We call $X\subseteq M$ an \emph{IP set} if there is a sequence $(u_n)$ for which $\FP(u_n)\subseteq X$, in which case we refer to $(u_n)$ as a \emph{basis} for $X$.
\end{defn}

\section{Main results}
 In this section,  $(M,\cdot)$ denotes an arbitrary countable semigroup. 
%
%
%

\begin{statement}[Hindman's theorem for $(M,\cdot)$]\label{MHindman}
Let $X\subseteq M$ be an IP-set. Then for any $Y\subseteq X$, either $Y$ or $X\setminus Y$ is an IP-set.
\end{statement}
\begin{statement}[Existence of idempotent types for semigroup structures based on $(M,\cdot)$]\label{Midempotents}
If $\m=(M,\cdot,\ldots)$ is a semigroup structure based on $(M,\cdot)$ and $X\subseteq M$ is a $\m$-definable IP-set, then there is an idempotent type over $\m$ containing $X$.
\end{statement}

\begin{thm}
Statement \ref{MHindman} is equivalent to Statement \ref{Midempotents}. 
\end{thm}

\begin{proof}[Proof that Statement \ref{MHindman} implies Statement \ref{Midempotents}:]
Let $(\varphi_i(x) \ : \ i<\omega)$ enumerate all $L(M)$-formulae in the free variable $x$ and let $(\psi_j(x,y) \ : j<\omega)$ enumerate all the $L(M)$-formulae in the free variables $x,y$.  Without loss of generality, assume that $\varphi_0$ is the formula that defines $X$.

We build an approximation to an independent type $q$ in stages. At every stage $s$, we build two finite sets of formulae:  $A_s(x)$ and $B_s(x,y)$.  Throughout the construction, we will maintain the following recursive assumptions:
\begin{enumerate}
\item For each $i\leq s$, exactly one of $\varphi_i$ or $\neg \varphi_i$ belongs to $A_s$;
\item $A_s(x)$ defines an IP subset of $M$;
\item There is $J_s\subseteq \{0,1,\ldots,s-1\}$ such that $B_s=\{\neg \psi_j \ : \ j\in J_s\}$.
\item For every $u\in M$ and every $j\in J_s$, $A_s(y)\cup \{\psi_j(u,y)\}$ does not define an IP subset of $M$.
\item For each $j\in \{0,1,\ldots,s-1\}\setminus J_s$, there is $u\in M$ so that $\psi_j(u,y)\in A_s(y)$.
\end{enumerate}

At stage $0$, we begin with $A_0(x)=\{\phi_0(x)\}$ and $B_0= \emptyset$.  Clearly (1)-(5) are satisfied at this stage.


Now assume that the construction has been carried out through stage $s$ and we show how to carry it through to stage $s+1$.  First, if $A_s(x)\cup \{\phi_{s+1}(x)\}$ defines an IP subset of $M$, then we set $A_{s+1}^0:=A_s(x)\cup\{\phi_{s+1}(x)\}$. Otherwise, we set $A_{s+1}^0:=A_s(x)\cup \{\neg \phi_{s+1}(x)\}$. By Statement \ref{MHindman}, in either case, $A_{s+1}^0$ defines an IP subset of $M$.

Now we shift our attention to $\psi_s(x,y)$. If there is $u\in M$ so that $A_{s+1}^0(y)\cup \{\psi_s(u,y)\}$ defines an IP subset of $M$, then we set $A_{s+1}(y)$ to be this set of formulae (for the least such $u$ with respect to some fixed ordering of $M$). Otherwise, set $A_{s+1}(y):=A_{s+1}^0(y)$ and put $\neg\psi_s(x,y)$ into $B_{s+1}$.  It is clear that (1)-(5) holds for $A_{s+1}$ and $B_{s+1}$. 



\noindent \textbf{Claim:}   For each $s$, $A_s(x)\cup A_s(y)\cup A_s(x\cdot y)\cup B_s(x,y)$ is consistent. 

\noindent \textbf{Proof of claim:}  Set $u$ to be the least element of $M$ (again, with respect some fixed ordering of $M$) that lies in a basis for the set defined by $A_s$.  It suffices to show that the set of formulae $$A_s(y)\cup A_s(u\cdot y)\cup B_s(u,y)$$ is consistent.  Let $Z$ denote the subset of $M$ defined by $A_s(y)\cup A_s(u\cdot y)$.  Since $u$ belongs to a basis for the set defined by $A_s$, it follows that $Z$ is an IP set.  By (4), $Z\cap \psi_j(u,M)$ is not an IP set for each $j\in J_s$.  By Statement \ref{MHindman}, we get that $Z\cap \bigcap_{j\in J_s} \neg \psi_j(u,M)$ is an IP set, whence is nonempty, proving the claim.   

By the claim, $q_0(x,y):=\bigcup_s (A_s(x)\cup A_s(y)\cup A_s(x\cdot y)\cup B_s(x,y))$ is a partial type over $M$.  By (1), the restriction of $q_0$ to the variable $x$ is a complete type $p(x)$ over $M$ and $p(x),p(y),p(x\cdot y)\subseteq q_0(x,y)$.  Let $q(x,y)\in S_2(M)$ be any completion of $q_0$.  It remains to prove that $q$ is independent.  Towards this end, fix $\theta(x,y)\in q$.  Take $s$ such that $\theta=\psi_s$.  Since $\neg\psi_s\notin B_{s+1}$, we have that $s\notin J_{s+1}$, so by (5), there is $u\in M$ such that $\psi_s(u,y)\in A_{s+1}(y)$, whence $\theta(u,y)\in q$.

\end{proof}

\begin{proof}[Proof that Theorem \ref{Midempotents} implies Theorem \ref{MHindman}:]
Fix an IP set $X\subseteq M$ and fix $Y\subseteq X$. Let $\mathcal{L}$ denote the language by $\{\cdot,X,Y\}$ and consider the semigroup structure $\m:=(M,\cdot, X,Y)$. Let $p(x)$ be an idempotent type contained in the independent type $q(x,y)$ containing the formula $X(x)$. Without loss of generality, we may assume $Y(x)$ belongs to $p$ (otherwise re-name $Y$ to define $X\smallsetminus Y$).

Set $\psi_1(x,y):=Y(x)\wedge Y(x\cdot y)$. Since $q$ witnesses that $p$ is idempotent, we have that $\psi_1(x,y)\in q$. Since $q$ is independent, there is a $u_1\in M$ so that $\psi_1(u_1,y)\in q$. Again, since $q$ witnesses that $p$ is idempotent, $Y(u_1\cdot x)\wedge Y(u_1\cdot x\cdot y)\in q$.

Let $\psi_2(x,y):=\psi_1(x,y)\wedge \psi_1(u_1\cdot x, y)$. Since $\psi_2(x,y)$ belong to $q$, there is $u_2\in M$ so that $\psi_2(u_2,y)\in p$. We now have that $u_1,u_2, u_1\cdot u_2\in Y$. Moreover, $\psi_2(u_2,x)\wedge \psi_2(u_2,x\cdot y)\in q$ Continuing in this manner, we construct a sequence $(u_i\mid i\in \omega)$ which is a basis for $Y$.
\end{proof}

\section{Musings on non-principality}

As mentioned above, in weird semigroups, IP sets can be finite, even singletons.  Likewise, idempotent types can be principal.  We mention here some conditions on semigroups that remove some of these trivialities.  

Here are two possible ways of making the notion of IP less trivial.

\begin{definition}
Suppose that $(M,\cdot)$ is a semigroup and $A\subseteq M$.
\begin{enumerate}
\item We say that $A$ is IIP (\emph{infinite IP}) if there is a sequence $(x_n)$ such that $\FP(x_n)\subseteq A$ and $\FP(x_n)$ is infinite.
\item We say that $A$ is DIP (\emph{distinctly IP}) if there is an injective sequence $(x_n)$ with $\FP(x_n)\subseteq A$.
\end{enumerate}
\end{definition}

Clearly DIP sets are IIP.  A class of semigroups where DIP is a good notion can be found in the literature:

\begin{definition}(Golan and Tsaban, \cite{golan})
We call a semigroup $(M,\cdot)$ \emph{moving} if $\beta M\setminus M$ is a subsemigroup of $\beta M$.
\end{definition}

There is a more combinatorial definition of moving semigroup, but let us be content with the ultrafilter definition.

\begin{lemma}
If $(M,\cdot)$ is moving, then $A\subseteq M$ is DIP if and only if there is a \emph{nonprincipal} idempotent ultrafilter $\u$ on $S$ containing $A$. 
\end{lemma}

\begin{proof}
If $A$ is DIP as witnessed by $(x_n)$, then $T:=\bigcap_{n=m}^\infty (\overline{\FP(x_n)_{n=m}^\infty}\cap (\beta S\setminus S))$ is a nonempty compact subsemigroup of $\beta S$.  If $\u\in T$ is idempotent, then $A\in \u$.  The converse follows from the usual argument, using the fact that one can always find a fresh element at every stage of the construction.
\end{proof}

\begin{corollary}
In moving semigroups, the notion of being DIP is partition regular.
\end{corollary}

Observe that in a moving semigroup, to conclude that $A$ belonged to a nonprincipal idempotent ultrafilter, all that was really used was that $A$ was IIP.  It thus follows that:

\begin{corollary}
In moving semigroups, the notions IIP and DIP coincide.
\end{corollary}

Here is an admittedly ad hoc definition:

\begin{definition}
We call a semigroup $(M,\cdot)$ \emph{Hindman} if the notion of being IIP is partition regular.
\end{definition}

It follows from the above corollaries that moving semigroups are Hindman.

The following theorem follows immediately from the proofs in the preceding section:

\begin{theorem}
Let $(M,\cdot)$ be a semigroup.
\begin{enumerate}
\item Suppose that $(M,\cdot)$ is Hindman and $\m$ is a semigroup structure based on $(M,\cdot)$.  Then for every $\m$-definable $X\subseteq M$ that is IIP, there is a \emph{nonprincipal} idempotent type containing the formula $X(x)$.
\item Suppose that for every semigroup structure $\m$ based on $(M,\cdot)$ and every $\m$-definable $X\subseteq M$ that is IIP, there is a nonprincipal idempotent type containing $X(x)$.  Then $(M,\cdot)$ is \emph{really} Hindman, meaning that whenever $X\subseteq M$ is IIP and $X=Y\cup Z$, then one of $Y$ or $Z$ is DIP.  
\end{enumerate}
\end{theorem}

\begin{corollary}
In Hindman semigroups, the notions IIP and DIP coincide.
\end{corollary}

\begin{question}
Do the notions IIP and DIP coincide in every semigroup?  Does the property that the DIP sets are partition regular characterize moving semigroups?  Is every Hindman semigroup moving?
\end{question}

A positive answer to the second question yields a positive answer to the third question.


\begin{thebibliography}{10}

\bibitem{BHS} Andreas Blass, Jeffry Hirst, and Stephen Simpson, \emph{Logical analysis of some theorems of combinatorics and topological dynamics}, Logic and combinatorics, Contemp. Math., 65, Amer. Math. Soc., Providence, RI, 1987

\bibitem{dinasso} M. Di Nasso, \textit{Hypernatural numbers as ultrafilters}, to appear as a chapter in ``Nonstandard Analysis for the Working mathematician'' (P.A. Loeb and M. Wolff, eds.), 2nd edition.  arXiv 1501.05755

\bibitem{golan} G. Golan and B. Tsaban, \emph{Hindman's coloring theorem in arbitrary semigroups}, Journal of Algebra 395 (2013), 111-120.

\bibitem{tao} T. Tao, \texttt{https://terrytao.wordpress.com/2012/11/28/multiple-recurrence-in-quasirandom-groups/}
\bibitem{townser} H. Towsner, \textit{A simple proof and some difficult examples for Hindman's theorem}, Notre Dame J. Formal Logic, \textbf{53} (2012), 53-65.
\end{thebibliography}
\end{document}